\documentclass[11pt,a4paper]{article}
\usepackage[T1]{fontenc}
\usepackage[utf8]{inputenc}
\usepackage[english]{babel}
\usepackage{lmodern}
\usepackage{amsmath,amssymb,amsthm,mathtools}
\usepackage[margin=23mm]{geometry}
\usepackage[unicode,hidelinks]{hyperref}
\usepackage[expansion=false]{microtype}
\hypersetup{
  pdftitle={Optimization Inverse Spectral Problem for the Schr\"odinger Operator on a Compact Graph},
  pdfauthor={Nur Valeev},
  pdfsubject={Realizability of spectral data, optimal potentials,
    independence of squared eigenfunctions, and a limiting representation}}
\numberwithin{equation}{section}
\theoremstyle{plain}
\newtheorem{theorem}{Theorem}[section]
\newtheorem{lemma}[theorem]{Lemma}
\newtheorem{proposition}[theorem]{Proposition}
\newtheorem{corollary}[theorem]{Corollary}
\theoremstyle{definition}

\theoremstyle{remark}
\newtheorem{remark}[theorem]{Remark}

\newcommand{\R}{\mathbb R}
\newcommand{\CC}{\mathbb C}
\newcommand{\HH}{\mathcal H}
\newcommand{\DD}{\mathcal D}
\newcommand{\Q}{\mathcal Q}
\newcommand{\Lq}{\mathcal L}
\newcommand{\ISM}{\operatorname{ISM}}
\newcommand{\Span}{\operatorname{span}}
\newcommand{\sn}{\operatorname{sn}}
\newcommand{\cn}{\operatorname{cn}}
\newcommand{\dn}{\operatorname{dn}}
\newcommand{\norm}[1]{\left\lVert #1\right\rVert}

\newcommand{\ip}[2]{\left\langle #1,#2\right\rangle}
\newcommand{\argmin}{\operatorname*{arg\,min}}

\title{Optimization Inverse Spectral Problem\protect\\
for the Schr\"odinger Operator on a Compact Graph}
\author{Nur Valeev}
\date{October 22, 2026}

\begin{document}
\maketitle

\begin{abstract}
We study the optimization inverse spectral problem (OISP) for the
first $m$ eigenvalues of the Schr\"odinger operator on a compact metric
graph. We prove the existence of a solution for every strictly
increasing finite set of prescribed eigenvalues. On a graph without
cycles, we obtain a representation of the optimal correction as a
linear combination of squared eigenfunctions, provided that the selected
eigenvalues are simple and separated from the next eigenvalue. On a
graph with cycles, under the same simplicity and separation assumptions,
we obtain a strong limiting representation of the optimal potential
in terms of eigenfunctions of approximating operators.

For a graph with a cycle, we give a counterexample with simple
eigenvalues and linearly dependent squared eigenfunctions. We also prove
that linear dependence of the squares of nonzero restrictions to an
edge of eigenfunctions corresponding to distinct eigenvalues forces
the potential on that edge to have an analytic representative.
\end{abstract}


\section{Introduction}

The Schr\"odinger operator is one of the basic models of mathematical
physics and spectral theory. In the stationary problem, its eigenvalues
are interpreted as energy levels and its eigenfunctions as the
corresponding states. Similar spectral models arise in the theory of
vibrations, wave propagation, and stability. On a metric graph, the
operator acts on each edge by the differential expression
$\mathcal L[q]u=-u''+qu$; the potential $q$ describes a field or an
inhomogeneity of the medium, while the vertex conditions specify the
interaction between the channels.

The classical inverse spectral problem seeks to recover the
coefficients of an operator and the parameters of its boundary
conditions from a sufficiently complete set of spectral
characteristics. Even for the one-dimensional Sturm--Liouville
operator, a single spectrum does not, in general, determine the
potential uniquely. Borg's theorem uses two spectra corresponding to
different boundary conditions \cite{Borg}. Other classical data include
the spectrum together with norming constants, the spectral function,
and the Weyl--Titchmarsh function. The Gel'fand--Levitan and Marchenko
methods reduce reconstruction to integral equations
\cite{GL,Marchenko}; the geometry of spectral data and isospectral sets
is studied in detail in \cite{PT}. The central questions remain the
sufficiency of the data, uniqueness, the characterization of
admissible data, and the stability of reconstruction.

In experimental and computational problems, only a few of the lowest
energy levels or eigenfrequencies are often available, while a prior
model specifies a reference potential $q_0$. Reconstruction from
finite spectral data was studied, in particular, in \cite{Lowe}.
The requirement of minimal deviation from the reference model
provides a natural variational principle for selecting a potential
among those with the observed part of the spectrum. In the present
paper, for a fixed compact graph $\Gamma$ and data
$\Lambda=(\lambda_1^*,\ldots,\lambda_m^*)$, we seek
\[
 \widehat q\in\argmin\left\{\tfrac12\norm{q-q_0}_{L^2(\Gamma)}^2:
 q\in L^2(\Gamma;\R),\quad
 \lambda_k(q)=\lambda_k^*,\ k=1,\ldots,m\right\}.
\]
This \emph{optimization inverse spectral problem} (OISP) is equivalent
to finding the point nearest to $q_0$ in a level set of the spectral
map $q\mapsto(\lambda_1(q),\ldots,\lambda_m(q))$. Until regularity is
established, we call it an isospectral set; smoothness is not assumed
a priori. Even strict convexity of the objective does not guarantee
uniqueness of the nearest point, since the spectral constraints are
nonlinear.

The OISP admits two complementary interpretations in applications.
In identification, the reference model is corrected as little as
possible to reproduce measured frequencies or levels. In spectral
design, prescribed spectral characteristics are achieved by a minimal
change in the potential. Inverse problems in vibration mechanics
are used to recover mass and stiffness characteristics and to detect
damage \cite{Gladwell}. Metric graphs model quantum and wave networks
\cite{BK}; an example of an engineering inverse problem on a graph is
the diagnosis of inhomogeneities in an electrical network from
reflection coefficients \cite{Visco}. Below we study agreement with
exact eigenvalues; stability with respect to measurement errors
requires a separate analysis.

The optimization approach for the one-dimensional Sturm--Liouville
operator was considered by N.~F.~Valeev and Y.~Sh.~Ilyasov in
\cite{VI2018}. V.~A.~Sadovnichii, Ya.~T.~Sultanaev, and
N.~F.~Valeev studied matrix and vector formulations
\cite{SSVvector,SSVspace}, as well as the one-dimensional
Schr\"odinger operator on the entire real line \cite{SSVline}.
The relation between the multidimensional OISP and a nonlinear
boundary value problem was established in \cite{IVN}.
Recovery of the nearest potential from $m$ observed eigenvalues
was studied in \cite{IVm}, and the problem with a prescribed partial
trace in \cite{VITrace}. In these formulations, the existence and
structure of the nearest element of the isospectral set become
questions in their own right.

An additional mathematical motivation comes from nonlinear
equations. For a simple eigenvalue and a real normalized
eigenfunction $\varphi_k$, the first variation is
$D\lambda_k(q)h=\int_\Gamma h\varphi_k^2$. When the spectral
constraints are independent, the optimality condition expresses
the correction in terms of squared eigenfunctions. Substitution
of this formula into the eigenvalue equations gives a coupled
system of nonlinear Schr\"odinger equations; this relation was
studied in \cite{VI2018,IVN,IVm}. Further developments of the
approach to nonlinear problems include work on degenerate
$(p,q)$-Laplace equations \cite{IVpq} and on principal solutions
of defocusing Hartree equations \cite{ISVYHartree}.

A related direction consists of extremal spectral problems.
Q.~Wei, G.~Meng, and M.~Zhang studied extremal eigenvalues of
the Sturm--Liouville operator on balls in $L^1$
\cite{WeiMengZhang}. Polynomial integrability of the critical
systems for optimal eigenvalue gaps was investigated by
Y.~Tian, Q.~Wei, and M.~Zhang \cite{TianWeiZhang}.
M.~Zhao and J.~Qi considered minimization of the norm of a
potential with two prescribed first eigenvalues \cite{ZhaoQi}.
These works emphasize the role of nonlinear critical equations
in spectral optimization.

The theory of inverse problems on graphs takes account of data
at different boundary vertices and of the presence of cycles.
General formulations for Sturm--Liouville operators were studied
by V.~A.~Yurko \cite{Yurko}. In the works of N.~P.~Bondarenko
\cite{BondStar} and C.-F.~Yang and N.~P.~Bondarenko
\cite{BondLoop}, part of the potential is known a priori, and
the remaining part is recovered from specially chosen spectral
data. S.~A.~Buterin studied recovery of operators with global
delay on graphs from two spectra \cite{Buterin}; this is a
functional differential model. Below we consider a local
self-adjoint operator with fixed geometry and vertex conditions.
For the OISP on a graph, both realizability of finite data and
independence of squared eigenfunctions, which ensures regularity
of the constraints, are essential.

The aim of this paper is to justify the variational selection
of a potential and to determine the precise form of the
necessary optimality condition. We prove realizability of every
strictly increasing finite set of eigenvalues and existence of
a nearest potential. On a tree, squares of eigenfunctions
corresponding to distinct eigenvalues are independent; on a
circle, a Lam\'e example shows that simplicity of the eigenvalues
alone is insufficient. On an arbitrary fixed graph, independence
is established as a generic property: separately in the sense
of Baire category and of Lebesgue measure in constructed
finite-dimensional families. For a local minimizer with simple
selected eigenvalues, without assuming independence, we prove
a strong limiting representation in terms of eigenfunctions
of approximating operators. We use the standard theory of
quantum graphs and analytic perturbations from \cite{BK,Kato};
proofs of the main assertions are given below.

\section{The model and the optimization problem}\label{sec:model}

Let $\Gamma$ be a finite connected compact metric graph with at
least one edge; an edge $e$ is identified with $[0,\ell_e]$,
$\ell_e>0$. Loops and multiple edges are allowed.
The geometry of the graph and all vertex conditions are fixed.
Set
\[
 X=L^2(\Gamma;\R),\qquad \HH=L^2(\Gamma;\CC).
\]
We impose Dirichlet conditions at selected vertices of degree
one, forming a set $D$. At all other vertices, functions are
continuous and satisfy the $\delta$-condition
\begin{equation}\label{eq:vertex}
 \sum_{e\sim v}\partial_e u(v)=\alpha_v u(v),\qquad
 \alpha_v\in\R.
\end{equation}
Here the derivative is directed from the vertex into the edge:
$u'_e(0)$ or $-u'_e(\ell_e)$; the ends of a loop are counted
separately. For $\alpha_v=0$, these are the Kirchhoff conditions.

On the space $\Q$ of continuous functions on the graph belonging
to $\bigoplus_eH^1(0,\ell_e)$ and vanishing on $D$, consider the form
\begin{equation}\label{eq:form}
 \mathfrak a_q[u,w]=\sum_e\int_0^{\ell_e}
 (u'_e\overline{w'_e}+q_eu_e\overline{w_e})\,dx
 +\sum_{v\notin D}\alpha_vu(v)\overline{w(v)},\qquad q\in X.
\end{equation}
It is closed and bounded from below. The associated operator
$\Lq(q)u=-u''+qu$ is self-adjoint, has compact resolvent, and
has a common domain
$\DD_0\subset\bigoplus_eH^2(0,\ell_e)$ for all $q\in X$,
specified by the above vertex conditions. These standard
properties and the min--max principle follow from form theory
and one-dimensional embeddings \cite{BK,Kato}.
Denote the eigenvalues, counted with multiplicities, by
$\lambda_1(q)\leq\lambda_2(q)\leq\cdots\to+\infty$.

For a fixed $m\geq1$, set
\begin{equation}\label{eq:problem}
 \begin{gathered}
 F_m(q)=(\lambda_1(q),\ldots,\lambda_m(q)),\qquad
 \ISM_m(\Lambda)=\{q\in X:F_m(q)=\Lambda\},\\
 J(q)=\tfrac12\norm{q-q_0}_X^2\longrightarrow\min,
 \qquad q\in\ISM_m(\Lambda),\quad q_0\in X.
 \end{gathered}
\end{equation}
Below, the term \emph{isospectral set} (ISM) refers to fixing
the first $m$ eigenvalues. Fixing the full spectrum gives an
intersection of such sets over all $m$; its nonemptiness does
not follow automatically from finite-dimensional results.
Admissible ordered data lie in
$\Delta_m=\{\Lambda\in\R^m:\lambda_1^*\leq\cdots\leq\lambda_m^*\}$;
denote its interior by
$\Omega_m=\{\lambda_1^*<\cdots<\lambda_m^*\}$.

\begin{proposition}[Weak continuity of the spectral data]
\label{prop:weak}
Let $q_j,q\in X=L^2(\Gamma;\R)$ and
$q_j\rightharpoonup q$ in $X$. Then, for every fixed $k\geq1$,
\begin{equation}\label{eq:weak-spectral-continuity}
 \lambda_k(q_j)\longrightarrow\lambda_k(q).
\end{equation}
The eigenvalues are numbered in nondecreasing order, counting
multiplicities; no simplicity assumptions are required.
In particular, $F_m(q_j)\to F_m(q)$ in $\R^m$ for every fixed $m$.
\end{proposition}

\begin{proof}
Equip the common form domain $\Q$ with the norm
\[
 \norm{u}_{\Q}^2=
 \sum_e\bigl(\norm{u_e}_{L^2(0,\ell_e)}^2
               +\norm{u'_e}_{L^2(0,\ell_e)}^2\bigr)
 =\norm{u}_{\HH}^2+\norm{u'}_{\HH}^2.
\]
Set $h_j=q_j-q$. Weak convergence means that
$\ip{h_j}{f}_X\to0$ for every $f\in X$.
By the uniform boundedness principle,
\begin{equation}\label{eq:weak-potential-bound}
 M:=\sup_j\norm{h_j}_X<\infty.
\end{equation}
The geometry of the graph and its vertex conditions do not
depend on $j$, so the forms $\mathfrak a_{q_j}$ and
$\mathfrak a_q$ have the same domain $\Q$.

\emph{1.}
For $w\in H^1(0,\ell)$, choose a point $y\in[0,\ell]$ at which
$|w(y)|^2\leq\ell^{-1}\norm{w}_{L^2}^2$. Such a point exists by
continuity of $w$ and the definition of the mean value.
Since $|w|^2$ is absolutely continuous and
$(|w|^2)'=2\operatorname{Re}(w'\overline w)$, for every
$x\in[0,\ell]$ we have
\[
 |w(x)|^2\leq |w(y)|^2+2\int_0^\ell|w'||w|
 \leq\ell^{-1}\norm{w}_{L^2}^2
                      +2\norm{w'}_{L^2}\norm{w}_{L^2}.
\]
Young's inequality gives, for every $\eta>0$,
\[
 \norm{w}_\infty^2\leq
 \eta\norm{w'}_{L^2}^2+(\ell^{-1}+\eta^{-1})\norm{w}_{L^2}^2.
\]
Applying this estimate on all edges yields
\begin{equation}\label{eq:graph-trace-estimate}
 \norm{u}_{C(\Gamma)}^2\leq
 \eta\norm{u'}_{\HH}^2+C_\eta\norm{u}_{\HH}^2,
 \qquad
 C_\eta=\max_e\ell_e^{-1}+\eta^{-1}.
\end{equation}
Here $\norm{u}_{C(\Gamma)}=\max_e\norm{u_e}_\infty$;
the vertex values $u(v)$ are also bounded by this norm.

Set $L=\sum_e\ell_e$. Since $q\in X$,
$\norm{q}_{L^1(\Gamma)}\leq\sqrt L\norm{q}_X$.
Thus the number
\[
 A_q=\norm{q}_{L^1(\Gamma)}+\sum_{v\notin D}|\alpha_v|
\]
is finite, and for $u\in\Q$,
\[
 \left|\int_\Gamma q|u|^2+
             \sum_{v\notin D}\alpha_v|u(v)|^2\right|
 \leq A_q\norm{u}_{C(\Gamma)}^2.
\]
Choose $\eta=[2(A_q+1)]^{-1}$ and set $C_q=A_qC_\eta$.
Then $A_q\eta\leq1/2$, and \eqref{eq:form} and
\eqref{eq:graph-trace-estimate} give
\[
 \mathfrak a_q[u,u]\geq
 \tfrac12\norm{u'}_{\HH}^2-C_q\norm{u}_{\HH}^2.
\]
For $c=C_q+1$, the shifted form
$b_q[u]=\mathfrak a_q[u,u]+c\norm{u}_{\HH}^2$ satisfies
\begin{equation}\label{eq:shifted-coercivity}
 b_q[u]\geq\tfrac12\norm{u'}_{\HH}^2+\norm{u}_{\HH}^2
          \geq\gamma\norm{u}_{\Q}^2,
 \qquad \gamma=\tfrac12.
\end{equation}
The constants $c,\gamma$ are fixed and independent of $j$.

\emph{2.}
We show that the embedding $\Q\hookrightarrow C(\Gamma)$
is compact. By \eqref{eq:graph-trace-estimate}, functions in
the unit ball of $\Q$ are uniformly bounded on all edges.
Moreover, on each edge, the Cauchy--Schwarz inequality gives
\[
 |u_e(x)-u_e(y)|\leq
 |x-y|^{1/2}\norm{u'_e}_{L^2(0,\ell_e)},
 \qquad x,y\in[0,\ell_e].
\]
Hence this ball is equicontinuous on each edge.
By the Arzel\`a--Ascoli theorem, every sequence in it has a
subsequence converging uniformly on all edges simultaneously,
since the number of edges is finite. Continuity at the vertices
and vanishing on $D$ are preserved under uniform convergence.
The limit therefore belongs to $C(\Gamma)$, proving compactness
of the embedding.

The map $u\mapsto|u|^2$ is continuous from $C(\Gamma)$ to $X$,
because
\[
 \bigl\||u|^2-|w|^2\bigr\|_X
 \leq\sqrt L\bigl(\norm{u}_{C(\Gamma)}+
                    \norm{w}_{C(\Gamma)}\bigr)
                \norm{u-w}_{C(\Gamma)}.
\]
Consequently, the set
\begin{equation}\label{eq:compact-squares}
 \mathcal K=\{|u|^2:u\in\Q,\ \norm{u}_{\Q}\leq1\}
\end{equation}
is relatively compact in the real space $X$.
Using $|u|^2$ also allows us to work with complex functions
in the form domain.

\emph{3.}
Set
\[
 \varepsilon_j=\sup_{f\in\mathcal K}|\ip{h_j}{f}_X|.
\]
For every $\delta>0$, compactness of the closure of
$\mathcal K$ allows us to choose a finite $\delta$-net
$f_1,\ldots,f_N\in X$. For every $f\in\mathcal K$, there is
an $f_\nu$ with $\norm{f-f_\nu}_X<\delta$, and
\[
 |\ip{h_j}{f}_X|
 \leq |\ip{h_j}{f_\nu}_X|+\norm{h_j}_X\norm{f-f_\nu}_X
 \leq\max_{1\leq\nu\leq N}|\ip{h_j}{f_\nu}_X|+M\delta.
\]
The first term on the right tends to zero, since the net is
finite and $h_j\rightharpoonup0$. Thus
$\limsup_j\varepsilon_j\leq M\delta$ for every $\delta>0$.
Letting $\delta$ tend to zero, we conclude that
\begin{equation}\label{eq:uniform-form-convergence}
 \varepsilon_j=
 \sup_{\norm{u}_{\Q}\leq1}
 \left|\int_\Gamma(q_j-q)|u|^2\right|\longrightarrow0.
\end{equation}
By homogeneity, for every $u\in\Q$,
\begin{equation}\label{eq:form-difference-estimate}
 |\mathfrak a_{q_j}[u,u]-\mathfrak a_q[u,u]|
 \leq\varepsilon_j\norm{u}_{\Q}^2
 \leq\theta_j b_q[u],\qquad
 \theta_j=\varepsilon_j/\gamma\longrightarrow0.
\end{equation}
The vertex terms cancel in the difference of the forms,
since all $\alpha_v$ are fixed.

\emph{4.}
Set $b_{q_j}[u]=\mathfrak a_{q_j}[u,u]+c\norm{u}_{\HH}^2$,
with the same $c$ as in the first step.
For sufficiently large $j$, we have $0\leq\theta_j<1$;
\eqref{eq:form-difference-estimate} implies
\begin{equation}\label{eq:ordered-shifted-forms}
 (1-\theta_j)b_q[u]\leq b_{q_j}[u]
                    \leq(1+\theta_j)b_q[u],\qquad u\in\Q.
\end{equation}
Both shifted forms are positive. By the min--max principle,
\[
 \lambda_k(q)+c=
 \inf_{\substack{E\subset\Q\\\dim_{\CC}E=k}}
 \ \sup_{u\in E\setminus\{0\}}
          \frac{b_q[u]}{\norm{u}_{\HH}^2},
\]
and the analogous formula holds for $q_j$.
Divide \eqref{eq:ordered-shifted-forms} by
$\norm{u}_{\HH}^2$, take the supremum over nonzero $u\in E$,
and then the infimum over all $k$-dimensional subspaces
$E\subset\Q$. Since $1-\theta_j>0$, we obtain
\begin{equation}\label{eq:eigenvalue-sandwich}
 (1-\theta_j)(\lambda_k(q)+c)
 \leq\lambda_k(q_j)+c
 \leq(1+\theta_j)(\lambda_k(q)+c).
\end{equation}
The quantity $\lambda_k(q)+c$ is positive and fixed.
It follows that
\[
 |\lambda_k(q_j)-\lambda_k(q)|
 \leq\theta_j(\lambda_k(q)+c)\longrightarrow0.
\]
This proves \eqref{eq:weak-spectral-continuity}.
The min--max principle numbers eigenvalues with multiplicities,
so the proof also applies to multiple eigenvalues.
Componentwise convergence of the first $m$ eigenvalues gives
convergence of $F_m(q_j)$ in $\R^m$.
\end{proof}

\begin{remark}
	There is an alternative approach to Proposition~\ref{prop:weak}
	based on distributional potentials. Writing the potential
	edgewise as $q=\sigma'$, weak convergence in $L^2$ implies
	strong convergence of the corresponding primitives in $L^2$.
	Continuity of the eigenvalues can then be obtained from the
	continuity of the characteristic function. For an interval,
	this approach goes back to Savchuk and Shkalikov~\cite{SavchukShkalikov};
	quantitative stability results for characteristic functions on
	quantum trees were obtained by Bondarenko~\cite{BondarenkoStability}.
	The direct form argument given above applies uniformly to arbitrary
	compact metric graphs, including graphs with cycles.
\end{remark}

\begin{proposition}[Existence of an OISP solution]
\label{prop:existence}
Let $q_0\in X$, $\Lambda\in\Delta_m$, and
$\ISM_m(\Lambda)\ne\varnothing$. Then there exists
$\widehat q\in\ISM_m(\Lambda)$ such that
\begin{equation}\label{eq:attained-minimum}
 J(\widehat q)=\min_{q\in\ISM_m(\Lambda)}J(q).
\end{equation}
\end{proposition}

\begin{proof}\emph{1.}
Choose any $\widetilde q\in\ISM_m(\Lambda)$ and set
\[
 \beta=\inf_{q\in\ISM_m(\Lambda)}J(q).
\]
Since $J\geq0$ and $J(\widetilde q)<\infty$,
$0\leq\beta\leq J(\widetilde q)<\infty$.
By the definition of the infimum, for every $n\geq1$ there is
a $q_n\in\ISM_m(\Lambda)$ with
\[
 \beta\leq J(q_n)\leq\beta+1/n.
\]
Thus $J(q_n)\to\beta$, and
\begin{equation}\label{eq:minimizing-sequence-bound}
 \norm{q_n-q_0}_X^2=2J(q_n)\leq2(\beta+1),\qquad
 \norm{q_n}_X\leq\norm{q_0}_X+\sqrt{2(\beta+1)}.
\end{equation}
Hence the sequence is bounded in the potential space $X$.

\emph{2.}
The space $X$ is Hilbert and reflexive, so a bounded sequence
has a subsequence $q_{n_j}$ and a limit $\widehat q\in X$
such that $q_{n_j}\rightharpoonup\widehat q$ in $X$.
Each $q_{n_j}$ is admissible, that is,
$\lambda_k(q_{n_j})=\lambda_k^*$ for $k=1,\ldots,m$.
By Proposition~\ref{prop:weak},
\[
 \lambda_k(\widehat q)=\lim_{j\to\infty}\lambda_k(q_{n_j})
                     =\lambda_k^*,\qquad k=1,\ldots,m.
\]
Therefore $\widehat q\in\ISM_m(\Lambda)$.
This preservation of the spectral constraints under weak
convergence is precisely what the direct variational method
requires; it does not require choosing eigenfunctions or
establishing their convergence.

\emph{3.}
Set $w_j=q_{n_j}-q_0$ and $w=\widehat q-q_0$.
We have $w_j\rightharpoonup w$ in $X$, and the identity
\[
 \norm{w_j}_X^2=
 \norm{w}_X^2+2\ip{w}{w_j-w}_X+\norm{w_j-w}_X^2
\]
gives $\norm{w}_X^2\leq\liminf_j\norm{w_j}_X^2$:
the middle term tends to zero and the last term is nonnegative.
Thus weak lower semicontinuity of the objective follows directly
from the Hilbert space structure of $X$, and
\[
 J(\widehat q)\leq\liminf_{j\to\infty}J(q_{n_j})=\beta.
\]
On the other hand, admissibility of $\widehat q$ and the
definition of $\beta$ imply $J(\widehat q)\geq\beta$.
We obtain $J(\widehat q)=\beta$, proving
\eqref{eq:attained-minimum}.
\end{proof}

\section{Nonemptiness of the isospectral set}

\begin{theorem}\label{thm:realization}
For every $\Lambda\in\Omega_m$, there exists
$q\in L^\infty(\Gamma;\R)$ such that $F_m(q)=\Lambda$ and
$\lambda_m(q)<\lambda_{m+1}(q)$. Consequently,
$\ISM_m(\Lambda)\ne\varnothing$ for almost every
$\Lambda\in\Delta_m$ with respect to $m$-dimensional Lebesgue measure.
\end{theorem}
\begin{proof}
Fix the data $\lambda_1^*<\cdots<\lambda_m^*$.
The idea is to place $m$ short potential wells inside a single
edge. Their ground-state eigenvalues will be close to the
prescribed numbers, while a high barrier outside the wells
will nearly decouple the corresponding Dirichlet problems.
A small correction to the well depths will then produce exact
agreement with the spectral data.

\emph{1.}
For $m\geq2$, choose $\delta>0$ such that
\[
 4\delta<\min_{1\leq j<m}
                  (\lambda_{j+1}^*-\lambda_j^*).
\]
For $m=1$, take any $\delta>0$.
Fix an edge $e_0$ of length $\ell_{e_0}>0$.
Choose $a>0$ sufficiently small that $(0,\ell_{e_0})$
contains $m$ open intervals $I_j=(b_j,b_j+a)$ whose closures
are pairwise disjoint and separated from one another and from
the endpoints of the edge by positive distances, and that
\begin{equation}\label{eq:well-excited-gap}
 \lambda_1^*-\delta+\frac{3\pi^2}{a^2}
                           >\lambda_m^*+3\delta.
\end{equation}
Both requirements can be met for sufficiently small $a$.
Set
\[
 U=\bigcup_{j=1}^m I_j,\qquad V=\Gamma\setminus U,\qquad
 T=[-\delta,\delta]^m,\qquad
 p_j=\lambda_j^*-\frac{\pi^2}{a^2}.
\]
For $t=(t_1,\ldots,t_m)\in T$, consider the direct sum of the
operators $-d^2/dx^2+p_j+t_j$ on $I_j$, with Dirichlet
conditions at both endpoints. Its Hilbert space and form
domain are, respectively,
\[
 \HH_\infty=\bigoplus_{j=1}^m L^2(I_j;\CC),\qquad
 \Q_\infty=\bigoplus_{j=1}^m H^1_0(I_j;\CC).
\]
Identify their elements with their zero extensions to the
graph. Under this identification, $\Q_\infty\subset\Q$,
since the extensions are continuous, vanish at all vertices,
and satisfy the conditions on $D$. The limiting form is
\[
 \mathfrak a_{\infty,t}[u]
   =\sum_{j=1}^m\int_{I_j}
              \bigl(|u'|^2+(p_j+t_j)|u|^2\bigr)\,dx.
\]
The spectrum of each well is explicit:
\[
 \nu_{j,n}(t)=p_j+t_j+\frac{\pi^2n^2}{a^2}
       =\lambda_j^*+t_j+\frac{\pi^2(n^2-1)}{a^2},
 \qquad n=1,2,\ldots .
\]
Its ground-state eigenvalues
$\nu_{j,1}(t)=\lambda_j^*+t_j$ are strictly ordered:
the difference between adjacent eigenvalues is at least
$\lambda_{j+1}^*-\lambda_j^*-2\delta>2\delta$.
By \eqref{eq:well-excited-gap}, all other eigenvalues exceed
$\lambda_m^*+3\delta$. Hence, if $\mu_k(t)$ denotes the
eigenvalues of the direct sum in nondecreasing order, counting
multiplicities, then, for every $t\in T$,
\begin{equation}\label{eq:decoupled-well-levels}
 \mu_k(t)=\lambda_k^*+t_k\quad(1\leq k\leq m),
 \qquad \mu_{m+1}(t)>\lambda_m^*+3\delta.
\end{equation}

\emph{2.}
For $B\geq0$, define the piecewise constant potential
\begin{equation}\label{eq:barrier-potential}
 q_{B,t}(x)=
 \begin{cases}
    p_j+t_j,&x\in I_j,\quad 1\leq j\leq m,\\
    B,&x\in V.
 \end{cases}
\end{equation}
The values at individual endpoints of the intervals are
irrelevant. For every finite $B$, this potential belongs to
$L^\infty(\Gamma;\R)\subset X$. Write
\[
 \lambda_k(B,t)=\lambda_k(q_{B,t}),\qquad
 \mathfrak a_{B,t}[u]=\mathfrak a_{q_{B,t}}[u,u].
\]
The form domain $\Q$ is independent of $B$ and $t$, and
\[
 \mathfrak a_{B,t}[u]
 =\norm{u'}_{\HH}^2
   +\sum_{j=1}^m(p_j+t_j)\norm{u}_{L^2(I_j)}^2
   +B\norm{u}_{L^2(V)}^2
   +\sum_{v\notin D}\alpha_v|u(v)|^2.
\]
We obtain an estimate uniform in $B\geq0$ and $t\in T$.
Set $P=\max_j|p_j|+\delta$ and
$A=\sum_{v\notin D}|\alpha_v|$.
The integral term over the wells is bounded from below by
$-P\norm{u}_{\HH}^2$. From \eqref{eq:graph-trace-estimate},
with $\eta=[2(A+1)]^{-1}$, we obtain
\[
 \sum_{v\notin D}\alpha_v|u(v)|^2
 \geq -A\eta\norm{u'}_{\HH}^2-A C_\eta\norm{u}_{\HH}^2.
\]
Thus, with a constant $C=P+A C_\eta\geq0$,
\begin{equation}\label{eq:uniform-barrier-bound}
 \mathfrak a_{B,t}[u]\geq
 \tfrac12\norm{u'}_{\HH}^2
   +B\norm{u}_{L^2(V)}^2-C\norm{u}_{\HH}^2,
 \qquad u\in\Q.
\end{equation}
The constant $C$ depends on the chosen wells, the data, and
the vertex conditions, but not on $B$ or $t$.

\emph{3.}
We first obtain an upper bound. The first $k$ eigenfunctions
of the decoupled problem, extended by zero to $V$, span a
$k$-dimensional subspace of $\Q$. On this subspace,
$\mathfrak a_{B,t}$ agrees with $\mathfrak a_{\infty,t}$:
the barrier term and all vertex terms vanish.
The min--max principle therefore gives
\begin{equation}\label{eq:barrier-upper-bound}
 \lambda_k(B,t)\leq\mu_k(t),\qquad B\geq0.
\end{equation}
Moreover, for $B_2\geq B_1$,
\[
 \mathfrak a_{B_2,t}[u]-\mathfrak a_{B_1,t}[u]
       =(B_2-B_1)\norm{u}_{L^2(V)}^2\geq0.
\]
By the min--max principle, $\lambda_k(B,t)$ is nondecreasing
in $B$. Estimate \eqref{eq:uniform-barrier-bound} implies
$\lambda_k(B,t)\geq-C$. Hence the finite limit
\[
 \zeta_k(t)=\lim_{B\to+\infty}\lambda_k(B,t)
                         \leq\mu_k(t)
\]
exists.

We prove the reverse inequality. Fix $k\geq1$ and a sequence
$B_n\to+\infty$. For every $n$, choose eigenfunctions
$u_{1,n},\ldots,u_{k,n}$, orthonormal in $\HH$, corresponding
to the first $k$ eigenvalues of the operator with potential
$q_{B_n,t}$. They need not depend continuously on $B_n$;
in eigenspaces of multiple eigenvalues, choose any orthonormal
basis. By \eqref{eq:uniform-barrier-bound} and
\eqref{eq:barrier-upper-bound},
\begin{equation}\label{eq:barrier-eigenfunction-bound}
 \tfrac12\norm{u_{r,n}'}_{\HH}^2
       +B_n\norm{u_{r,n}}_{L^2(V)}^2
 \leq\lambda_r(B_n,t)+C\leq\mu_k(t)+C,
 \qquad 1\leq r\leq k.
\end{equation}
Thus these functions are bounded in $\Q$, and their
$L^2$-norms on $V$ tend to zero. Weak compactness of bounded
sets in $\Q$ and the compact embedding
$\Q\hookrightarrow C(\Gamma)$, proved in
Proposition~\ref{prop:weak}, allow us to pass to a common
subsequence for all $r=1,\ldots,k$ such that
\[
 u_{r,n}\rightharpoonup u_r\ \text{in }\Q,\qquad
 u_{r,n}\to u_r\ \text{in }C(\Gamma)\ \text{and in }\HH.
\]
The limits remain orthonormal:
$\ip{u_r}{u_s}_{\HH}=\delta_{rs}$.
Estimate \eqref{eq:barrier-eigenfunction-bound} shows that
$u_r=0$ almost everywhere on $V$. By continuity, the limits
also vanish at the endpoints of every $I_j$ and at all
vertices; here we use the positive distances separating the
closures of the wells from one another and from the vertices.
Consequently, $u_r\in\Q_\infty$.

For an arbitrary $z=(z_1,\ldots,z_k)\in\CC^k$, set
\[
 w_n=\sum_{r=1}^k z_r u_{r,n},\qquad
 w=\sum_{r=1}^k z_r u_r.
\]
Then $w_n\rightharpoonup w$ in $\Q$ and $w_n\to w$ in
$C(\Gamma)$ and $\HH$. Weak lower semicontinuity of the
derivative norm gives
\[
 \norm{w'}_{\HH}^2\leq\liminf_{n\to\infty}
                                      \norm{w_n'}_{\HH}^2.
\]
The integral terms over the wells converge to the corresponding
terms for $w$, since the coefficients $p_j+t_j$ are bounded
and $w_n\to w$ in $\HH$. The vertex terms tend to zero by
uniform convergence and the equalities $w(v)=0$.
Finally, the barrier term $B_n\norm{w_n}_{L^2(V)}^2$
is nonnegative. Therefore
\[
 \mathfrak a_{\infty,t}[w]
            \leq\liminf_{n\to\infty}\mathfrak a_{B_n,t}[w_n].
\]
Orthonormality and the eigenvalue equations give
\[
 \mathfrak a_{B_n,t}[w_n]
      =\sum_{r=1}^k|z_r|^2\lambda_r(B_n,t).
\]
Passing to the limit yields
\begin{equation}\label{eq:barrier-limit-subspace}
 \mathfrak a_{\infty,t}[w]
   \leq\sum_{r=1}^k|z_r|^2\zeta_r(t)
   \leq\zeta_k(t)\norm{w}_{\HH}^2.
\end{equation}
The subspace
$E=\Span\{u_1,\ldots,u_k\}\subset\Q_\infty$ has dimension $k$.
Applying the min--max principle in $\HH_\infty$, we conclude
from \eqref{eq:barrier-limit-subspace} that
\[
 \mu_k(t)\leq
   \sup_{0\ne w\in E}
        \frac{\mathfrak a_{\infty,t}[w]}{\norm{w}_{\HH}^2}
   \leq\zeta_k(t)\leq\mu_k(t).
\]
All inequalities are equalities. Thus
\begin{equation}\label{eq:barrier-pointwise-convergence}
 \lambda_k(B,t)\longrightarrow\mu_k(t)
       \quad(B\to+\infty)
\end{equation}
for every fixed $k$ and $t$.

\emph{4.}
For $s,t\in T$, potentials with the same barrier satisfy
\[
 \norm{q_{B,t}-q_{B,s}}_{L^\infty(\Gamma)}
                  \leq\norm{t-s}_\infty.
\]
Consequently, for every $u\in\Q$,
\[
 |\mathfrak a_{B,t}[u]-\mathfrak a_{B,s}[u]|
                 \leq\norm{t-s}_\infty\norm{u}_{\HH}^2.
\]
Applying the min--max principle in both directions gives
\begin{equation}\label{eq:barrier-parameter-lipschitz}
 |\lambda_k(B,t)-\lambda_k(B,s)|
                       \leq\norm{t-s}_\infty.
\end{equation}
The same estimate holds for $\mu_k(t)$ on $\Q_\infty$.
In particular, both functions of the parameter are continuous,
with a Lipschitz constant independent of $B$.

Let $\varepsilon>0$. The compact cube $T$ has a finite
$\varepsilon/4$-net $t^{(1)},\ldots,t^{(N)}$ in the norm
$\norm{\cdot}_\infty$. By
\eqref{eq:barrier-pointwise-convergence}, there is a $B_0$
such that, for $B\geq B_0$, all points of the net, and
all $k=1,\ldots,m+1$,
\[
 |\lambda_k(B,t^{(\nu)})-\mu_k(t^{(\nu)})|
                                      <\varepsilon/2.
\]
For an arbitrary $t\in T$, choose a point of the net with
$\norm{t-t^{(\nu)}}_\infty<\varepsilon/4$.
The Lipschitz estimates for both eigenvalues then give
\[
 |\lambda_k(B,t)-\mu_k(t)|
 \leq 2\norm{t-t^{(\nu)}}_\infty
       +|\lambda_k(B,t^{(\nu)})-\mu_k(t^{(\nu)})|
 <\varepsilon.
\]
Hence
\begin{equation}\label{eq:barrier-uniform-convergence}
 \max_{1\leq k\leq m+1}\
 \sup_{t\in T}|\lambda_k(B,t)-\mu_k(t)|
                                 \longrightarrow0.
\end{equation}

\emph{5.}
Choose and fix a finite $B\geq0$ sufficiently large that the
maximum in \eqref{eq:barrier-uniform-convergence} is smaller
than $\delta/2$. By \eqref{eq:decoupled-well-levels},
\begin{equation}\label{eq:barrier-residual}
 F_m(q_{B,t})=\Lambda+t+r_B(t),\qquad
 \sup_{t\in T}\norm{r_B(t)}_\infty<\delta/2.
\end{equation}
The map $r_B:T\to\R^m$ is continuous by
\eqref{eq:barrier-parameter-lipschitz}. Thus
\[
 \Psi_B(t)=-r_B(t)
\]
maps the compact convex cube $T$ continuously into itself,
indeed into $[-\delta/2,\delta/2]^m$.
By Brouwer's theorem, there is a $t^*\in T$ with
$t^*=\Psi_B(t^*)=-r_B(t^*)$.
Substituting this equality into \eqref{eq:barrier-residual}
gives $F_m(q_{B,t^*})=\Lambda$.

We verify separately the gap after the $m$th eigenvalue.
For every $t\in T$, \eqref{eq:decoupled-well-levels}
and the chosen uniform estimate imply
\[
 \lambda_{m+1}(B,t)
   >\mu_{m+1}(t)-\delta/2
   >\lambda_m^*+2\delta.
\]
In particular, for $q=q_{B,t^*}$,
$\lambda_m(q)=\lambda_m^*<\lambda_{m+1}(q)$.
The potential $q$ is bounded, since
$\norm{q}_{L^\infty}\leq\max\{B,P\}<\infty$.
This proves the first assertion of the theorem.

\emph{6. }
Every strictly ordered tuple $\Lambda\in\Omega_m$ is realized
by the constructed potential, so
$\ISM_m(\Lambda)\ne\varnothing$.
For $m\geq2$, the set $\Delta_m\setminus\Omega_m$
is contained in the union of $m-1$ hyperplanes
\[
 \bigcup_{j=1}^{m-1}
       \{\Lambda\in\R^m:\lambda_j^*=\lambda_{j+1}^*\}.
\]
Each hyperplane has zero $m$-dimensional Lebesgue measure,
and so does their finite union. For $m=1$,
$\Omega_1=\Delta_1=\R$, and the exceptional set is empty.
Thus the isospectral set is nonempty for almost every
tuple of data in $\Delta_m$.
\end{proof}
\begin{remark}
	For special star graphs, related finite realizability statements
	can be derived from the inverse spectral characterizations obtained
	by Pivovarchik~\cite{PivovarchikSimple,PivovarchikStar} and
	Bondarenko~\cite{BondarenkoCharacterization,
		BondarenkoMatrix}. Those results concern the completion of spectral
	data satisfying appropriate asymptotic and compatibility conditions.
	
	Theorem~\ref{thm:realization} has a different scope. It concerns an
	arbitrary fixed compact metric graph, possibly containing cycles and
	having arbitrary edge lengths and fixed vertex conditions. It directly
	realizes any prescribed strictly increasing finite tuple as the first
	$m$ eigenvalues and simultaneously guarantees the spectral gap
	$\lambda_m(q)<\lambda_{m+1}(q)$.
\end{remark}
The theorem does not settle the question for all tuples with
multiple eigenvalues or the admissibility of an arbitrary
full spectrum.

\section{Representation of an OISP solution for the Schr\"odinger operator
on a graph without cycles}

Set
\[
 \mathcal S_m=\{q:\lambda_1(q)<\cdots<\lambda_m(q)
                         <\lambda_{m+1}(q)\}.
\]
On $\mathcal S_m$, choose real normalized eigenfunctions
$\varphi_k(q)$. Locally, the eigenvalues and eigenfunctions
are analytic in $q\in X$; the eigenfunctions can be chosen
analytically with values in $\DD_0$, and their signs do not
affect their squares. The first variation formula
\cite{Kato,Roehrl} is
\begin{equation}\label{eq:HF}
 D\lambda_k(q)h=\ip{h}{g_k(q)}_X,
 \quad g_k(q)=\varphi_k(q)^2,
 \quad DF_m(q)^*a=\sum_{k=1}^m a_kg_k(q).
\end{equation}
Denote the Gram matrix by
$G(q)=(\ip{g_i(q)}{g_j(q)}_X)_{i,j=1}^m$ and set
\begin{equation}\label{eq:regularset}
 d_m(q)=\det G(q),\qquad
 \mathcal R_m=\{q\in\mathcal S_m:d_m(q)>0\}.
\end{equation}
Regularity, that is, surjectivity, of
$DF_m(q):X\to\R^m$ is equivalent to linear independence of
$g_1(q),\ldots,g_m(q)$.

Ilyasov and Valeev \cite{IVm} proved an assertion on the
independence of squared eigenfunctions of the Schr\"odinger
operator. Here, using and developing the idea of that work,
we prove an analogous assertion.

\begin{theorem}[Independence of squared eigenfunctions on graph edges]
\label{thm:edge}
Let $n\geq1$, $\ell>0$, $q\in L^2(0,\ell;\R)$, and let
$\mu_1,\ldots,\mu_n\in\R$ be pairwise distinct.
Let real functions $y_k\in H^2(0,\ell)$, $y_k\not\equiv0$,
satisfy
\[
 -y_k''+qy_k=\mu_ky_k
       \quad\text{almost everywhere on }(0,\ell),\qquad k=1,\ldots,n.
\]
Suppose that, at the same endpoint $x_*\in\{0,\ell\}$,
all functions satisfy a common Dirichlet or a common Robin
condition:
\[
 \begin{aligned}
  &y_k(x_*)=0
       &&\text{for all }k\quad\text{(Dirichlet)},\\
  &y_k'(x_*)=\beta y_k(x_*)
       &&\text{for all }k\quad\text{(Robin)},
 \end{aligned}
\]
where, in the Robin case, $\beta\in\R$ is the same for all $k$.
Then $y_1^2,\ldots,y_n^2$ are linearly independent in
$L^2(0,\ell;\R)$.
\end{theorem}
\begin{proof}
We present the proof in a form applicable to a single edge
and verify the boundary step for a common Robin condition
as well.

\emph{1.}
In one dimension, the embedding
$H^1(0,\ell)\hookrightarrow C[0,\ell]$ is continuous.
Thus $y_k\in H^2(0,\ell)$ implies $y_k\in C^1[0,\ell]$;
$y_k$ and $y_k'$ are bounded, and $y_k'$ is absolutely
continuous. The product rule for weak derivatives gives,
almost everywhere,
\begin{equation}\label{eq:edge-products}
 \begin{aligned}
 (y_k^2)'&=2y_ky_k',\\
 (y_k^2)''&=2(y_k')^2+2y_ky_k'',\\
 ((y_k')^2)'&=2y_k'y_k''.
 \end{aligned}
\end{equation}
Since $y_k''\in L^2$ and $y_k,y_k'$ are bounded, all
right-hand sides involving the second derivative belong
to $L^2(0,\ell)$. In particular,
\[
 y_k^2\in H^2(0,\ell),\qquad (y_k')^2\in H^1(0,\ell).
\]
The traces $y_k(x_*)$ and $y_k'(x_*)$ are therefore defined.

Suppose that, for some $c_1,\ldots,c_n\in\R$,
\[
 \sum_{k=1}^n c_ky_k^2=0
               \quad\text{in }L^2(0,\ell).
\]
We prove that all $c_k$ vanish.
The sum on the left is continuous on $[0,\ell]$.
Its vanishing almost everywhere therefore implies vanishing
at every point, including the endpoints.

\emph{2.}
For every integer $r\geq0$, define
\begin{equation}\label{eq:edge-moments}
 S_r=\sum_{k=1}^n c_k\mu_k^r y_k^2,\qquad
 T_r=\sum_{k=1}^n c_k\mu_k^r(y_k')^2.
\end{equation}
We have $S_r\in H^2(0,\ell)$ and $T_r\in H^1(0,\ell)$.
We use their continuous representatives also at the
endpoints. By assumption, $S_0\equiv0$.

From the eigenvalue equations $y_k''=(q-\mu_k)y_k$
and \eqref{eq:edge-products}, we obtain
\[
 \begin{aligned}
 S_r''
 &=2\sum_{k=1}^n c_k\mu_k^r
                    \bigl((y_k')^2+y_ky_k''\bigr)\\
 &=2T_r+2\sum_{k=1}^n c_k\mu_k^r(q-\mu_k)y_k^2\\
 &=2T_r+2qS_r-2S_{r+1}.
 \end{aligned}
\]
Similarly,
\[
 \begin{aligned}
 T_r'
 &=2\sum_{k=1}^n c_k\mu_k^r y_k'y_k''\\
 &=2q\sum_{k=1}^n c_k\mu_k^r y_ky_k'
       -2\sum_{k=1}^n c_k\mu_k^{r+1}y_ky_k'\\
 &=qS_r'-S_{r+1}'.
 \end{aligned}
\]
Thus, for all $r\geq0$,
\begin{equation}\label{eq:ST}
 \begin{aligned}
 S_r''&=2T_r+2qS_r-2S_{r+1},\\
 T_r'&=qS_r'-S_{r+1}'.
 \end{aligned}
\end{equation}
These equalities hold almost everywhere, hence in the sense
of weak derivatives. The products $qS_r$ and $qS_r'$ belong
to $L^2$, since $S_r,S_r'$ are continuous and bounded.
The identities use only $q$ and the eigenvalue equations;
no smoothness of $q$ beyond membership in $L^2$ is required.

\emph{3.}
Suppose that $S_r\equiv0$ has been proved for some $r\geq0$.
Then $S_r'=S_r''=0$, and the first identity in
\eqref{eq:ST} gives
\begin{equation}\label{eq:edge-next-moment}
 T_r=S_{r+1}.
\end{equation}
This equality is initially obtained almost everywhere,
but both functions are continuous, so it holds on all
of $[0,\ell]$. The second identity in \eqref{eq:ST} gives
\[
 T_r'=-S_{r+1}'.
\]
On the other hand, differentiating
\eqref{eq:edge-next-moment} weakly gives $T_r'=S_{r+1}'$.
Comparing the equalities, we obtain $2S_{r+1}'=0$
almost everywhere. Since $S_{r+1}\in H^1(0,\ell)$
is absolutely continuous, it is constant.
By \eqref{eq:edge-next-moment}, $T_r$ is constant as well:
\begin{equation}\label{eq:edge-constant-moment}
 S_{r+1}(x)=T_r(x)=C_r,\qquad x\in[0,\ell].
\end{equation}

\emph{4.}
We verify that $C_r=0$ in both cases under consideration.

Under the Dirichlet condition, all $y_k(x_*)=0$, and hence
\[
 S_{r+1}(x_*)=\sum_{k=1}^n
                  c_k\mu_k^{r+1}y_k(x_*)^2=0.
\]
It follows from \eqref{eq:edge-constant-moment} that $C_r=0$.

Under the Robin condition, all $y_k'(x_*)=\beta y_k(x_*)$
with the same $\beta$. Therefore
\[
 \begin{aligned}
 T_r(x_*)
 &=\sum_{k=1}^n c_k\mu_k^r(y_k'(x_*))^2\\
 &=\beta^2\sum_{k=1}^n c_k\mu_k^r y_k(x_*)^2
   =\beta^2S_r(x_*)=0.
 \end{aligned}
\]
The last equality follows from the induction hypothesis.
Again, \eqref{eq:edge-constant-moment} gives $C_r=0$.
This includes the Neumann condition when $\beta=0$.
The boundary step uses only the traces of $y_k$ and $y_k'$,
whose existence was established in the first step.

We have proved the implication
$S_r\equiv0\Longrightarrow S_{r+1}\equiv0$.
Starting from $S_0\equiv0$, induction gives
\begin{equation}\label{eq:edge-vanishing-moments}
 S_r(x)=0\qquad
        \text{for all }r\geq0\text{ and }x\in[0,\ell].
\end{equation}

\emph{5.}
Fix an arbitrary $x\in[0,\ell]$.
Equalities \eqref{eq:edge-vanishing-moments}, for
$r=0,\ldots,n-1$, form the system
\[
 \sum_{k=1}^n \mu_k^r\,c_ky_k(x)^2=0,
                  \qquad r=0,\ldots,n-1.
\]
Its matrix
\[
 \mathcal V_\mu=(\mu_k^r)_{
             \substack{0\leq r\leq n-1\\1\leq k\leq n}}
\]
is a Vandermonde matrix, and
\[
 \det\mathcal V_\mu
       =\prod_{1\leq i<j\leq n}(\mu_j-\mu_i)\ne0,
\]
since the $\mu_k$ are pairwise distinct.
Thus the only solution of the system for the unknowns
$c_ky_k(x)^2$ is zero. Since $x$ was arbitrary,
\[
 c_ky_k(x)^2=0
       \quad\text{for all }x\in[0,\ell],\quad k=1,\ldots,n.
\]
Integration gives $c_k\norm{y_k}_{L^2(0,\ell)}^2=0$.
All $y_k$ are nonzero, so their squared norms are strictly
positive and $c_k=0$ for every $k$.
This proves linear independence.
\end{proof}

\begin{theorem}\label{thm:tree}
On a tree, for every $q\in X$, the squares of nonzero real
eigenfunctions corresponding to distinct eigenvalues are
linearly independent. In particular, $\mathcal R_m=\mathcal S_m$.
\end{theorem}
\begin{proof}
Call an edge active for an eigenfunction $u$ if $u$ is not
identically zero on it. The active edges form a nonempty
forest. A leaf of this forest must be a leaf of the entire
tree: otherwise, an inactive edge would be incident to it,
implying $u(v)=0$, while \eqref{eq:vertex} would give zero
derivative on the sole active edge. Uniqueness for the
Cauchy problem contradicts activity of that edge.
Thus $u$ is nonzero on some external edge.
Restrict a presumed dependence of the squares to each
external edge and apply Theorem~\ref{thm:edge} at its
external endpoint. The coefficient of each eigenfunction
is zero, since that function is active on at least one
such edge.
\end{proof}

\begin{theorem}\label{thm:finite}
Let $q^*\in\ISM_m(\Lambda)\cap\mathcal R_m$ be a local
minimizer of \eqref{eq:problem}. Then there exist unique
$a_1,\ldots,a_m\in\R$ such that
\begin{equation}\label{eq:finite}
 q^*=q_0-\sum_{k=1}^m a_k\varphi_k(q^*)^2.
\end{equation}
In particular, this formula holds on a tree when
$q^*\in\mathcal S_m$.
\end{theorem}
\begin{proof}
The implicit function theorem describes the admissible set
near $q^*$ as a submanifold with tangent space
$\ker DF_m(q^*)$. The minimum condition gives
$q^*-q_0\perp\ker DF_m(q^*)$, that is,
$q^*-q_0\in\operatorname{ran}DF_m(q^*)^*$.
Formula \eqref{eq:HF} and independence of the squares prove
existence and uniqueness of the multipliers.
\end{proof}

If $b_i=\ip{q^*-q_0}{g_i(q^*)}_X$, then $G(q^*)a=-b$.
The correction $q^*-q_0$ belongs to $\bigoplus_eH^2$
and is continuous on the graph; boundedness of $q^*$
itself follows only if $q_0$ is bounded.
For $q_0=0$, the formula represents the optimal potential
itself. Setting $u_k=\sqrt{|a_k|}\varphi_k(q^*)$ and
$\sigma_k=\operatorname{sgn}a_k$, we obtain the coupled system
\begin{equation}\label{eq:nonlinear}
 -u_k''+\left(q_0-\sum_j\sigma_ju_j^2\right)u_k
 =\lambda_k^*u_k,\qquad
 \ip{u_i}{u_j}_{\HH}=0\ (i\ne j),\quad
 \norm{u_k}_{\HH}^2=|a_k|,
\end{equation}
with the same vertex conditions. These are necessary
conditions: by themselves they establish neither
minimality nor the positions of the eigenvalues in the
ordered spectrum.
When the squares are dependent, the usual multiplier
rule may give only an abnormal condition
$\sum_ka_kg_k(q^*)=0$, which does not express the optimal
correction. A general limiting condition will be proved
in Section~\ref{sec:limit}.

\section{Linear dependence of squared eigenfunctions and generic linear independence}

\subsection{An example on a circle}

Consider a circle of length $4K(\kappa)$, $0<\kappa<1$,
with periodic conditions, where $K$ is the complete elliptic
integral of the first kind, and the Lam\'e potential
\[
 q(x)=2\kappa^2\sn^2(x,\kappa).
\]
The standard identities for Jacobi functions
\cite{KhareSukhatme} give eigenfunctions $\dn x$, $\cn x$,
and $\sn x$ at eigenvalues $\kappa^2$, $1$, and
$1+\kappa^2$, respectively. These are the first three
eigenvalues, each of which is simple. Indeed, the two
commuting reflections $x\mapsto-x$ and $x\mapsto2K-x$
decompose the periodic problem into a direct sum of four
problems on $[0,K]$ with conditions $NN,ND,DN,DD$.
The functions $\dn,\cn,\sn$ are positive in the interior
of this interval and are the ground-state eigenfunctions
of the $NN,ND,DN$ problems. Interlacing gives
\[
 \lambda_1^{DD}\geq\lambda_1^{DN},\qquad
 \lambda_2^{NN}\geq\lambda_1^{DN},\qquad
 \lambda_2^{ND},\lambda_2^{DN}\geq\lambda_1^{DD}.
\]
The first two inequalities are strict: the spectra of
$DD$ and $DN$, as well as those of $NN$ and $DN$, are
disjoint, since the common condition at one endpoint
makes the solutions proportional, while coincident
eigenvalues would require both $D$ and $N$ conditions
at the other endpoint. Hence all remaining eigenvalues
are greater than $1+\kappa^2$.

However, the identities $\cn^2+\sn^2=1$ and
$\dn^2+\kappa^2\sn^2=1$ imply
\begin{equation}\label{eq:lame}
 \dn^2x-\cn^2x-(1-\kappa^2)\sn^2x=0.
\end{equation}
If $I_d,I_c,I_s$ are the squared $L^2$-norms of these
functions on the circle, then, for the normalized
eigenfunctions, the dependence becomes
$I_d\varphi_1^2-I_c\varphi_2^2-(1-\kappa^2)I_s\varphi_3^2=0$.
The rank is two, since $\cn^2$ and $\sn^2$ are independent,
as follows from their values at $x=0,K$.
On a circle of any fixed length $L$, the same example
is obtained by replacing the potential with
$q_L(x)=2\gamma^2\kappa^2\sn^2(\gamma x,\kappa)$,
where $\gamma=4K(\kappa)/L$; all eigenvalues are multiplied
by $\gamma^2$. Thus even simple selected eigenvalues
on a graph with cycles may give a singular spectral map.

\begin{samepage}
\subsection{A local consequence of dependence}

\begin{lemma}\label{lem:rigidity}
Let $\ell>0$, $q\in L^2(0,\ell;\R)$, and let
$\mu_1,\ldots,\mu_n\in\R$, $n\geq1$, be pairwise distinct.
Let nonzero real functions $y_k\in H^2(0,\ell)$ satisfy
\[
 -y_k''+qy_k=\mu_ky_k
       \quad\text{almost everywhere on }(0,\ell),\qquad k=1,\ldots,n.
\]
If $y_1^2,\ldots,y_n^2$ are linearly dependent in
$L^2(0,\ell;\R)$, then, for some $\varepsilon>0$,
there exists an analytic function
\[
 \widetilde q:(-\varepsilon,\ell+\varepsilon)\to\R
\]
equal to $q$ almost everywhere on $(0,\ell)$.
No boundary conditions are imposed on the functions $y_k$.
\end{lemma}
\end{samepage}
\begin{proof}
\emph{1.}
Choose real coefficients of dependence $c_k$, not all
zero, such that
\[
 \sum_{k=1}^n c_ky_k^2=0\quad\text{in }L^2(0,\ell).
\]
Discard the functions with zero coefficients and renumber
the remaining functions, retaining $n$ for their number.
Now $c_k\ne0$ for every $k$, and $n\geq2$: if $n=1$,
the equality $c_1y_1^2=0$ would contradict
$y_1\not\equiv0$.
As shown in the first step of the proof of
Theorem~\ref{thm:edge}, $y_k^2\in H^2(0,\ell)$ and
$(y_k')^2\in H^1(0,\ell)$. The dependence therefore holds
at every point of $[0,\ell]$.

Define $S_r$ and $T_r$ by \eqref{eq:edge-moments}.
Then $S_0\equiv0$, $S_r\in H^2(0,\ell)$,
$T_r\in H^1(0,\ell)$, and the identities
\eqref{eq:ST} hold. Their derivation uses only the
eigenvalue equations and the rules for weak derivatives,
so no boundary conditions are needed.
In all subsequent pointwise equalities for $S_r$ and
$T_r$, we use their continuous representatives.

\emph{2.}
First observe that, for every $r\geq0$, $S_r\equiv0$
implies that $S_{r+1}$ is constant.
Indeed, by \eqref{eq:ST},
\[
 T_r=S_{r+1},\qquad T_r'=-S_{r+1}'.
\]
Differentiating the first equality weakly and comparing
it with the second yields $2S_{r+1}'=0$.
Absolute continuity of $S_{r+1}$ means that it is
constant on $[0,\ell]$. The value of this constant
has not yet been determined.

The functions $S_1,\ldots,S_{n-1}$ cannot all vanish
identically. Otherwise, together with $S_0\equiv0$,
they would give, at each $x\in[0,\ell]$, the system
\[
 \sum_{k=1}^n\mu_k^r c_ky_k(x)^2=0,
                   \qquad r=0,\ldots,n-1.
\]
Its Vandermonde matrix is invertible, since the $\mu_k$
are pairwise distinct. Hence $c_ky_k(x)^2=0$ for all
$x$ and $k$. As $c_k\ne0$, this would imply
$y_k\equiv0$, contrary to the assumption.

Let $p\in\{1,\ldots,n-1\}$ be the smallest index such
that $S_p\not\equiv0$. Then $S_{p-1}\equiv0$, and the
constancy proved above gives
\begin{equation}\label{eq:rigidity-moment}
 S_p(x)=C\ne0\qquad\text{for all }x\in[0,\ell].
\end{equation}

\emph{3.}
Substituting $S_p=C$ and $S_p'=S_p''=0$ into
\eqref{eq:ST} yields, almost everywhere,
\begin{equation}\label{eq:rigidity-identities}
 T_p=S_{p+1}-Cq,\qquad T_p'=-S_{p+1}'.
\end{equation}
By the second equality, the function
$T_p+S_{p+1}\in H^1(0,\ell)$ has zero weak derivative.
Thus there is a constant $A\in\R$ such that
$T_p+S_{p+1}=A$ on all of $[0,\ell]$.
Eliminating $T_p$ from this equality and the first
identity in \eqref{eq:rigidity-identities}, we obtain
\begin{equation}\label{eq:rigidity-potential}
 \begin{aligned}
 q&=q_c+\frac{2}{C}S_{p+1}
     =q_c+\sum_{k=1}^n b_ky_k^2
                 &&\text{almost everywhere on }(0,\ell),\\
 q_c&=-\frac{A}{C},\qquad
 b_k=\frac{2c_k\mu_k^{p+1}}{C}.
 \end{aligned}
\end{equation}
This formula is obtained without differentiating $q$
or making any a priori smoothness assumptions on it.
Its right-hand side belongs to $H^2(0,\ell)$;
denote its continuous representative by
$Q\in C^1[0,\ell]$. Then $Q=q$ almost everywhere.

\emph{4.}
Replacing $q$ with $Q$ in the original equations gives
\begin{equation}\label{eq:rigidity-system}
 y_k''=\left(q_c+\sum_{j=1}^n b_jy_j^2-\mu_k\right)y_k,
                      \qquad k=1,\ldots,n.
\end{equation}
Initially, these equalities hold almost everywhere.
Their right-hand sides are continuous on $[0,\ell]$,
since $y_j\in C^1[0,\ell]$. Absolute continuity of
$y_k'$ gives the integral identity
\[
 y_k'(x)=y_k'(0)+
   \int_0^x\left(q_c+\sum_{j=1}^n b_jy_j(t)^2-\mu_k\right)y_k(t)\,dt.
\]
Consequently, $y_k\in C^2[0,\ell]$, and
\eqref{eq:rigidity-system} is a classical system.

Set
\[
 Z=(y_1,\ldots,y_n,v_1,\ldots,v_n),\qquad v_k=y_k'.
\]
Write $y=(y_1,\ldots,y_n)$ and $v=(v_1,\ldots,v_n)$.
Then $Z\in C^1[0,\ell]$ satisfies the first-order system
\begin{equation}\label{eq:rigidity-first-order}
 \begin{aligned}
 Z'&=\mathcal P(Z),\\
 \mathcal P(y,v)&=\left(v,\;
   \left(\left(q_c+\sum_{j=1}^n b_jy_j^2-\mu_k\right)y_k
                  \right)_{k=1}^n\right).
 \end{aligned}
\end{equation}
All components of $\mathcal P$ are polynomials with
real coefficients, so $\mathcal P$ is analytic and
locally Lipschitz on $\R^{2n}$.

By the theorem on analyticity of solutions of ordinary
differential equations with analytic right-hand sides,
a local analytic solution passes through every finite
initial point. Uniqueness for the Cauchy problem
implies that this solution agrees with $Z$ wherever
both are defined. Applying this at every interior
point of the interval proves analyticity of $Z$
on $(0,\ell)$.

\emph{5.}
The vectors $Z(0)$ and $Z(\ell)$ are finite, because
$y_k\in H^2(0,\ell)\hookrightarrow C^1[0,\ell]$.
Solve the analytic Cauchy problems
\[
 Z_- '=\mathcal P(Z_-),\quad Z_-(0)=Z(0),\qquad
 Z_+ '=\mathcal P(Z_+),\quad Z_+(\ell)=Z(\ell).
\]
There are $\delta_->0$ and $\delta_+>0$ such that
these solutions are analytic on
$(-\delta_-,\delta_-)$ and
$(\ell-\delta_+,\ell+\delta_+)$, respectively.
By decreasing these numbers, we may assume that
$\delta_-,\delta_+<\ell/2$.
Uniqueness gives $Z_-=Z$ on $[0,\delta_-)$ and
$Z_+=Z$ on $(\ell-\delta_+,\ell]$.
The three solutions therefore agree on their
overlaps and define a single analytic extension
$\widetilde Z$ to $(-\delta_-,\ell+\delta_+)$.

Let $\varepsilon=\min\{\delta_-,\delta_+\}$, and let
$\widetilde y_k$ be the first $n$ components of
$\widetilde Z$. The function
\[
 \widetilde q(x)=q_c+\sum_{k=1}^n b_k\widetilde y_k(x)^2,
                 \qquad x\in(-\varepsilon,\ell+\varepsilon),
\]
is real analytic and, by \eqref{eq:rigidity-potential},
agrees with $q$ almost everywhere on $(0,\ell)$.
This is the required representative of the potential,
analytic in a neighborhood of the entire closed interval.
\end{proof}

In particular, if $q_e\notin H^2(0,\ell_e)$ on every
edge, then the squares of any finite collection of
eigenfunctions corresponding to distinct eigenvalues
are independent. Indeed, on each edge the lemma
excludes dependence among nonzero restrictions, and
every nonzero eigenfunction is nonzero on at least
one edge. Absence of an analytic representative of
the potential on every closed edge also suffices.

\subsection{Generic independence of squared eigenfunctions}

Generic independence of squared eigenfunctions was previously
studied, for example, under variation of the domain for the
Laplacian \cite{PrivatSigalotti}, while generic simplicity of
the spectrum of a graph was studied under variation of its
edge lengths \cite{BerkolaikoLiu}. Here the geometry and vertex
conditions are fixed; only the potential varies.

\begin{theorem}\label{thm:generic}
For any fixed graph and every $m$, the sets $\mathcal S_m$
and $\mathcal R_m$ are open and dense in $X$.
The set $\bigcap_{m\geq1}\mathcal R_m$ is a dense $G_\delta$:
for its potentials, the entire spectrum is simple and every
finite collection of squared eigenfunctions is independent.
The same assertions hold in
$Y=\prod_eC^\infty([0,\ell_e];\R)$ with the Fr\'echet topology.
\end{theorem}
\begin{proof}
\emph{Simplicity.}
For an eigenspace of multiplicity $r$, choose a real
orthonormal basis $\psi_1,\ldots,\psi_r$.
Analytic perturbation theory \cite{Kato} reduces the spectral
cluster to a symmetric matrix $A(q)$, with, at the
eigenvalue $\lambda$,
\begin{equation}\label{eq:cluster}
 A(q)=\lambda I,
 \qquad DA(q)h=M(h),\quad
 M(h)_{ij}=\int_\Gamma h\psi_i\psi_j.
\end{equation}
The vectors $v(x)=(\psi_1(x),\ldots,\psi_r(x))$ at
interior points of the edges span $\R^r$: a coefficient
vector orthogonal to all of them would give an
identically zero linear combination of the eigenfunctions
and would therefore be zero.
Choose $r$ independent vectors $v(x_j)$.
For suitable positive weights, the matrix
$\sum_jw_jv(x_j)v(x_j)^T$ has a simple positive spectrum.
This follows inductively: each sufficiently small new
summand increases the rank by one, creates one small
positive eigenvalue, and preserves the gaps between
the preceding positive eigenvalues.
Approximating the point masses by smooth functions
with small supports gives $h\in Y$ such that $M(h)$
has simple spectrum.

For finitely many clusters, we can choose a single
such $h$. Indeed, in the finite-dimensional span of
all the chosen perturbations, the discriminant of
each $M(h)$ is a nonzero polynomial; the product of
these polynomials is also not identically zero.
Then $A(q+th)=\lambda I+tM(h)+O(t^2)$ splits each
cluster for every sufficiently small $t\ne0$.
Take all clusters up to the last one containing
$\lambda_m(q)$, preserving the gap to the next cluster.
This proves density of $\mathcal S_m$ in $X$ and in
$Y$; openness follows from continuity of the eigenvalues.

\emph{Independence in $X$.}
Closed bounded balls of $H^2$ are compact and have
empty interior in $L^2$. Hence $H^2(0,\ell_e)$
is a countable union of closed nowhere dense subsets
of $L^2(0,\ell_e)$.
The set
\[
 \mathcal A=\{q\in X:q_e\notin H^2(0,\ell_e)
                              \text{ for every }e\}
\]
is a dense $G_\delta$ in $X$.
By Lemma~\ref{lem:rigidity},
$\mathcal A\cap\mathcal S_m\subset\mathcal R_m$,
which proves density of $\mathcal R_m$.
Its openness follows from continuity of the squares
in $X$ and of the Gram determinant.

\emph{Independence in $Y$.}
Let $q\in\mathcal S_m\cap Y$.
Take a small convex $X$-ball about $q$ contained in
$\mathcal S_m$. By openness and density of
$\mathcal R_m$ in $X$ and density of $Y$ in $X$,
this ball contains a $q_1\in\mathcal R_m\cap Y$.
The function $t\mapsto d_m(q+t(q_1-q))$ is analytic
in a neighborhood of $[0,1]$ and nonzero at $t=1$.
Its zeros are isolated; regular points therefore
occur for arbitrarily small $t$, and
$t(q_1-q)\to0$ in the Fr\'echet topology.
This proves density of $\mathcal R_m\cap Y$;
openness follows from the continuous embedding
$Y\hookrightarrow X$.
The final assertions follow from Baire's theorem.
\end{proof}

The following formulation gives a concrete
finite-dimensional meaning to almost every potential.

\begin{theorem}\label{thm:ae}
For every $\bar q\in X$ and $m$, there exist
$h\in Y$, $v\in X$, and $\varepsilon>0$ such that
\[
 \bar q+sh+tv\in\mathcal R_m
 \quad\text{for almost every }(s,t)\in(-\varepsilon,\varepsilon)^2
\]
with respect to two-dimensional Lebesgue measure.
\end{theorem}
\begin{proof}
On each edge, set
$v_e=\mathbf1_{(\ell_e/2,\ell_e)}\notin H^2(0,\ell_e)$,
and choose $h$ as in the simplicity proof, for all
clusters up to the last one containing
$\lambda_m(\bar q)$. In a small square, these
clusters remain isolated.
The discriminants of their analytic matrices for
$\bar q+sh+tv$ are not identically zero, since on
the line $t=0$ the clusters split for small $s\ne0$.
The zero set of a nonzero real analytic function
of two variables has measure zero.
For example, choosing a nonzero coefficient
function in the expansion in $s$ gives only a
discrete exceptional set of $t$; for all other $t$,
the zeros in $s$ are discrete, and Fubini's theorem
applies. Hence almost every parameter pair gives
a potential in $\mathcal S_m$.

For each fixed $s$, the condition
$\bar q_e+sh_e+tv_e\in H^2(0,\ell_e)$ can hold
for at most one $t$: two such values would imply
$v_e\in H^2$.
The number of edges is finite, and the exceptional
sets are measurable, since $H^2$ is an $F_\sigma$
subset of $L^2$. Again by Fubini's theorem,
almost every parameter pair gives a potential
in $\mathcal A$.
Lemma~\ref{lem:rigidity} completes the proof.
\end{proof}

\section{A limiting representation at a nonregular point of the isospectral set}
\label{sec:limit}

Dependence of squared eigenfunctions does not prevent
the following representation.
The eigenfunctions in it belong to approximating
potentials; replacing them with eigenfunctions of
$q^*$ itself requires separate justification.
This approach is related to asymptotic multiplier
conditions in infinite-dimensional optimization
\cite{BKMW}.

\begin{theorem}\label{thm:limit}
Let $q^*\in\ISM_m(\Lambda)\cap\mathcal S_m$ be a
local minimizer of \eqref{eq:problem}.
Choose $\rho_0>0$ such that $J(q)\geq J(q^*)$
for all admissible $q$ with
$\norm{q-q^*}_X<\rho_0$, and fix $0<\rho<\rho_0$.
For every $n\geq1$, a minimum exists:
\begin{equation}\label{eq:penalty}
 q_n\in\argmin_{\norm{q-q^*}_X\leq\rho}
 \left\{J(q)+\tfrac12\norm{q-q^*}_X^2
                 +\tfrac n2\norm{F_m(q)-\Lambda}_{\R^m}^2\right\}.
\end{equation}
Every choice of these minima satisfies
\begin{equation}\label{eq:stronglimit}
 q_n\to q^*\text{ in }X,\qquad
 J(q_n)\to J(q^*),\qquad
 n\norm{F_m(q_n)-\Lambda}_{\R^m}^2\to0.
\end{equation}
For all sufficiently large $n$, we have
$q_n\in\mathcal S_m$.
For these $n$, let $\varphi_{k,n}$ be real normalized
eigenfunctions of $\Lq(q_n)$, and set
$a_{k,n}=n(\lambda_k(q_n)-\lambda_k^*)$.
Then, for all sufficiently large $n$,
\begin{equation}\label{eq:stationary}
 \sum_{k=1}^m a_{k,n}\varphi_{k,n}^2=q_0+q^*-2q_n
 \quad\text{in }X,
\end{equation}
and
\begin{equation}\label{eq:limitrepresentation}
 \boxed{\displaystyle
 q^*=q_0-\lim_{n\to\infty}
                 \sum_{k=1}^m a_{k,n}\varphi_{k,n}^2,
 \qquad\text{strong limit in }L^2(\Gamma;\R).}
\end{equation}
For all sufficiently large $n$, the representation
has the exact error formula
\begin{equation}\label{eq:error}
 \left\|q^*-\left(q_0-\sum_{k=1}^m
                         a_{k,n}\varphi_{k,n}^2\right)\right\|_X
       =2\norm{q_n-q^*}_X,
 \qquad \norm{a_n}_{\R^m}=o(\sqrt n).
\end{equation}
\end{theorem}
\begin{proof}
The closed ball in \eqref{eq:penalty} is weakly
compact. By Proposition~\ref{prop:weak}, the
spectral part of the functional is continuous
along weakly convergent sequences, while the
squared norms are weakly lower semicontinuous.
The direct method gives a minimum $q_n$.
Comparison with the admissible point $q^*$ yields
\begin{equation}\label{eq:basicpenalty}
 J(q_n)+\tfrac12\norm{q_n-q^*}_X^2
 +\tfrac n2\norm{F_m(q_n)-\Lambda}_{\R^m}^2\leq J(q^*).
\end{equation}
Thus the spectral residual tends to zero.
We show that $\liminf J(q_n)\geq J(q^*)$.
Otherwise, a subsequence with
$J(q_n)\leq J(q^*)-\eta$, $\eta>0$,
has a weak limit $\bar q$ in this ball.
By spectral continuity, $F_m(\bar q)=\Lambda$;
since $\norm{\bar q-q^*}_X\leq\rho<\rho_0$,
local minimality gives $J(\bar q)\geq J(q^*)$.
This contradicts weak lower semicontinuity of $J$.
Inequality \eqref{eq:basicpenalty} now gives all
three limits in \eqref{eq:stronglimit} for the
entire sequence.

For large $n$, the point $q_n$ lies in the
interior of the ball and in the open set
$\mathcal S_m$.
The usual stationarity condition for the
differentiable penalty functional and
\eqref{eq:HF} give
\[
 (q_n-q_0)+(q_n-q^*)+
       \sum_{k=1}^m n(\lambda_k(q_n)-\lambda_k^*)
                           \varphi_{k,n}^2=0.
\]
This is \eqref{eq:stationary}; strong convergence
of $q_n$ proves \eqref{eq:limitrepresentation}
and the equality in \eqref{eq:error}.
Finally,
$\norm{a_n}/\sqrt n=\sqrt n\norm{F_m(q_n)-\Lambda}\to0$.
\end{proof}

For a global minimizer, minimization in
\eqref{eq:penalty} can be performed over all of $X$:
the functional is coercive, and the proof is
unchanged.
The additional term $\frac12\norm{q-q^*}^2$
selects the chosen minimizer and ensures strong
convergence. It involves the unknown $q^*$,
so this construction is a theoretical
representation rather than a ready-to-use
reconstruction algorithm based on the data.

\begin{corollary}\label{cor:multipliers}
If the multipliers $a_n$ in
Theorem~\ref{thm:limit} have a bounded subsequence,
then the finite representation \eqref{eq:finite}
holds for some $a\in\R^m$, in terms of eigenfunctions
of $q^*$ itself.
If $q^*\in\mathcal R_m$, the entire sequence $a_n$
converges to the unique multiplier of
Theorem~\ref{thm:finite}.
If $\norm{a_n}\to\infty$ along a subsequence,
then every limit of the normalized multipliers
$a_n/\norm{a_n}$ is a unit vector $c$ with
\begin{equation}\label{eq:abnormal}
 \sum_{k=1}^m c_k\varphi_k(q^*)^2=0.
\end{equation}
\end{corollary}
\begin{proof}
Local analyticity gives
$\varphi_{k,n}^2\to g_k(q^*)$ in $X$.
Along a bounded subsequence, take a further
subsequence with $a_n\to a$ and pass to the
limit in \eqref{eq:stationary}.
If $G(q^*)$ is positive definite, then, for
some $\gamma>0$ and all large $n$,
\[
 \gamma\norm{a_n}_{\R^m}^2
 \leq a_n^TG(q_n)a_n
 =\left\|\sum_k a_{k,n}\varphi_{k,n}^2\right\|_X^2.
\]
The right-hand side is bounded by
\eqref{eq:stationary}. The multipliers are
bounded, and uniqueness of the limiting
multiplier gives convergence of the entire
sequence.
In the unbounded case, divide
\eqref{eq:stationary} by $\norm{a_n}$ and use
compactness of the unit sphere in $\R^m$;
the right-hand side tends to zero, giving
\eqref{eq:abnormal}.
\end{proof}

When the squares are dependent, a finite
formula, if it exists, does not determine
the coefficients uniquely.
With $b_i=\ip{q^*-q_0}{g_i(q^*)}_X$ as before,
all coefficients are of the form
\[
 a=-G(q^*)^+b+c,\qquad c\in\ker G(q^*),
\]
where $G^+$ is the Moore--Penrose pseudoinverse,
and $-G^+b$ has the smallest Euclidean norm.
Without a previously established finite formula,
the pseudoinverse gives only the orthogonal
projection of the correction onto
$\Span\{g_1(q^*),\ldots,g_m(q^*)\}$, rather than
the entire correction.
Likewise, convergence of individual squares
with unbounded $a_n$ does not allow them to
be replaced in \eqref{eq:limitrepresentation}
by the squares at $q^*$: this would require
an estimate of the product of the
eigenfunction error and the multipliers.

\section{Conclusion}

For finite strictly ordered data, the OISP
on a fixed graph always has a solution.
On a tree, simple selected eigenvalues
ensure regularity and a finite formula
for the optimal correction.
On a graph with a cycle, regularity may
fail, but independence of the squares is
generic under variation of the potential.
For every local minimizer with simple
selected eigenvalues, the strong limiting
representation \eqref{eq:limitrepresentation}
remains valid.



\end{document}